\documentclass{article}
\usepackage[utf8]{inputenc}

\usepackage{amsmath}
\usepackage{amsfonts}
\usepackage{amssymb}
\usepackage{amsthm}
\usepackage{amsrefs}
\usepackage{tikz}
\usepackage{pdfpages}
\usepackage{enumerate}
\usepackage{geometry}

\usetikzlibrary{matrix,arrows,calc}

\DeclareMathOperator{\lcm}{lcm}
\DeclareMathOperator{\tor}{Tor}
\let \gcd \relax
\DeclareMathOperator{\gcd}{gcd}

\DeclareMathOperator{\N}{\mathbb{N}}
\DeclareMathOperator{\Z}{\mathbb{Z}}
\let \dim \relax
\DeclareMathOperator{\dim}{dim}
\let \min \relax
\DeclareMathOperator{\min}{min}
\DeclareMathOperator{\sgn}{sgn}

\DeclareMathOperator{\im}{im}

\DeclareMathOperator{\colim}{colim}

\theoremstyle{definition}
\newtheorem{definition}{Definition}[section]
\newtheorem{theorem}[definition]{Theorem}
\newtheorem{lemma}[definition]{Lemma}
\newtheorem{corollary}[definition]{Corollary}

\newtheorem{example}[definition]{Example}
\newtheorem{remark}[definition]{Remark}
\newtheorem*{question}{Question}

\newtheorem{maintheorem}{Theorem}

\begin{document}

\title{$A_{\infty}$-resolutions and the Golod property for monomial rings}
\author{Robin Frankhuizen}

\maketitle

\begin{abstract}
Let $R=S/I$ be a monomial ring whose minimal free resolution $F$ is rooted. 
We describe an $A_{\infty}$-algebra structure on $F$. Using this structure, we show that $R$ is Golod if and only if the product on $\tor^S(R,k)$ vanishes. 
Furthermore, we give a necessary and sufficient combinatorial condition for $R$ to be Golod.
\end{abstract}

\section{Introduction}
Let $S=k[x_1,\ldots,x_m]$ be the polynomial algebra over a field $k$ in $m$ variables and let $I = (m_1,\ldots, m_r)$ be an ideal generated by monomials.
In that case, $S/I$ is called a \emph{monomial} ring. Given a monomial ring $R = S/I$, the \emph{Poincar\'e series} of $R$ is defined as
\[
 P(R) = \sum_{j=0}^{\infty} \dim \tor^R_j(k,k) t^j.
\]
A result due to Serre states that there is an inequality of power series
\begin{equation*}
 P(R) \leq \frac{(1+t)^m}{1-t(\sum_{j=0}^{\infty}\dim \tor^S_j(R,k) t^j -1)}.
\end{equation*}
The ring $R$ is said to be \emph{Golod} if equality is obtained. The problem of when a monomial ring is Golod goes back to at least the 70s when Golod \cite{golod1978} showed that a monomial ring $R$ is Golod if and only if all Massey products on the Tor-algebra $\tor^S(R,k)$ vanish. 
In general, it is hard to directly verify the vanishing of Massey products and so in practice the Golod property is still hard to determine. 

In recent decades, the Golod property has received an increasing amount of attention in topology. The Tor-algebra shows up naturally in topology as follows. Let $\Delta$ be a simplicial complex on vertex set $[m] = \{1, \ldots, m \}$ and define the \emph{moment-angle complex} $Z_{\Delta}$ as follows. Let $D^2$ denote the $2$-disc and $S^1$ its bounding circle. For $\sigma \in \Delta$, define
$$X_{\sigma} = \prod_{i=1}^m Y_i \subseteq (D^2)^m \quad \mbox{ where } \quad Y_i = \begin{cases} D^2 &\mbox{ if } i \in \sigma \\ S^1 &\mbox{ if } i \notin \sigma \end{cases}$$
Lastly, we put
$$Z_{\Delta} = \colim_{\sigma \in \Delta} X_{\sigma} \subseteq (D^2)^m.$$
Moment-angle complexes are one of the central objects of study in toric topology. For us, the cohomology of $Z_{\Delta}$ is of particular interest.

\begin{theorem}[\cite{buchstaberpanov2015}, Theorem 4.5.4]
Let $\Delta$ be a simplicial complex. There is an isomorphism of graded algebras
$$H^*(Z_{\Delta},k) \cong \tor^S(k[\Delta],k).$$
\end{theorem}

Here, $k[\Delta]$ denotes the \emph{Stanley-Reisner ring} 
$$k[\Delta] = S / (x_{i_1}\cdots x_{i_k} \mid \{i_1,\ldots,i_k \} \notin \Delta )$$
of the simplicial complex $\Delta$. Note that $k[\Delta]$ is a square-free monomial ring. In general, the homotopy type of $Z_{\Delta}$ is not well understood, but significant progress has been made for those $Z_{\Delta}$ where $\Delta$ is Golod, see for example Grbi\'c and Theriault \cite{grbictheriault2007,grbictheriault2013}, Iriye and Kishimoto \cite{iriyekishimoto2013b} and Beben and Grbi\'c \cite{bebengrbic2017}.

The preceding discussion makes clear that the Golod property is of interest in both commutative algebra and algebraic topology. Consequently, a lot of work has been done on the Golodness problem. For example, a combinatorial characterization of Golodness in terms of the homology of the lower intervals in the lattice of saturated subsets is given by Berglund in \cite{berglund2006}. Using results from J\"ollenbeck \cite{jollenbeck2006}, it has been claimed in Berglund and J\"ollenbeck \cite{berglundjollenbeck2007} that $R$ is Golod if and only if the product on $\tor^S(R,k)$ vanishes. However, recently a counterexample to this claim was found by Katth\"an in \cite{katthan2015} where the error is traced back to \cite{jollenbeck2006}. This leads naturally to the central question this work investigates. 
\begin{question}
For which classes of monomial rings $R$ is the Golod property equivalent to the vanishing of the product on $\tor^S(R,k)$?
\end{question}
A partial answer to this question is given by Theorem \ref{maintheorem2} below. To answer this question, we develop a new approach to the Golodness problem using $A_{\infty}$-algebras. 
An $A_{\infty}$-algebra is similar to a differential graded algebra (dga), except that associativity only holds up coherent homotopy. 
By contrast with dgas, every resolution admits the structure of an $A_{\infty}$-algebra (as first shown by Burke \cite{burke2015}) hence in particular the minimal free resolution does. 
The first main result of this paper characterizes vanishing of Massey products in terms of this $A_{\infty}$-structure. A monomial ring $R$ is said to satisfy condition $(B_r)$ if all $k$-Massey products are defined and contain only zero for all $k \leq r$. Denote by $K_R$ the Koszul dga of the monomial ring $R$. We obtain the following result.

\begin{maintheorem}
 Let $R = S/I$ be a monomial ring with minimal free resolution $F$. Let $r \in \N$ and let $\mu_n$ be an $A_{\infty}$-structure on $F$ such that $F \otimes_S k$ and $K_R$ are quasi-isomorphic as $A_{\infty}$-algebras. Then $R$ satisfies $(B_r)$ if and only if $\mu_k$ is minimal for all $k \leq r$. 
\end{maintheorem}

Next, we turn our attention to the class of rooted rings. A monomial ring is said to be \emph{rooted} if the minimal free resolution $F$ of $R$ is rooted in the sense of Novik \cite{novik2002}. Rooted resolutions include both the Taylor and Lyubeznik resolution. Given a rooted ring with rooting map $\pi$, we give an explicit $A_{\infty}$-structure in terms of $\pi$. 
 
This $A_{\infty}$-structure allows us to give a combinatorial characterization of the Golod property for rooted rings as follows. Following \cite{jollenbeck2006}, we say that $R$ satisfies the \emph{gcd condition} if for all generators $m_i$ and $m_j$ with $\gcd(m_i,m_j)=1$ there exists a $m_k \neq m_i,m_j$ such that $m_k$ divides $\lcm(m_i,m_j)$. The second main result is then the following. 

\begin{maintheorem}
\label{maintheorem2}
 Let $R$ be a rooted ring. Then the following are equivalent.
 \begin{enumerate}
  \item The ring $R$ is Golod.
  \item The product on $\tor^S(R,k)$ vanishes.
  \item The ring $R$ is gcd. 
 \end{enumerate}
\end{maintheorem}

In particular, the main result from \cite{berglundjollenbeck2007} does hold when restricted to rooted rings.

\section{Simplicial resolutions}
Let $S=k[x_1,\ldots, x_m]$ and let $I$ be the ideal minimally generated by monomials $m_1, \ldots,m_r$. The \emph{Taylor resolution} $T$ \cite{taylor1966} of $S/I$ is constructed as follows. 
Let $E$ denote the exterior algebra on generators $u_1, \ldots, u_r$. The resolution $T$ has underlying module $S \otimes_k E$. 
If $J = \lbrace j_1 < \cdots < j_k \rbrace \subseteq \lbrace1, \ldots,r \rbrace$, then we write $u_J = u_{j_1} \cdots u_{j_k}$. Furthermore, we put $m_J = \lcm(m_{j_1}, \ldots, m_{j_k})$. 
We will also write $J^i = \lbrace j_1 < \cdots < \widehat{j_i} < \cdots < j_k \rbrace$. 
The differential $d$ of $T$ is give by
\[
 d(u_J) = \sum_{i=1}^{\vert J \vert} (-1)^{i+1} \frac{m_J}{m_{J^i}} u_{J^i}.
\]
The Taylor resolution admits a multiplication defined by
\[
 u_I \cdot u_J = \begin{cases} \sgn(I,J) \frac{m_Im_J}{m_{I\cup J}} u_{I \cup J} &\mbox{ if } I \cap J = \emptyset \\ 0 &\mbox{ otherwise} \end{cases}
\]
where $\sgn(I,J)$ is the sign of the permutation making $I \cup J$ into an increasing sequence.
This multiplication induces a differential graded algebra (dga) structure on $T$. The \emph{Tor-algebra} $\tor^S(S/I,k)$ of $S/I$ is 
$$ \tor^S(S/I,k) = \bigoplus_{n} \tor^S_n(S/I,k) = \bigoplus_n H_n(T \otimes_S k)$$
where the multiplication is induced by the multiplication on $T$\\

The following method of constructing free resolutions of monomial rings is due to Bayer, Peeva and Sturmfels \cite{bayerpeevasturmfels1998}. Our exposition will follow that of Mermin \cite{mermin2012}. 
Let $\lbrace m_1, \ldots, m_r \rbrace$ be a set of monomials. Fix some total order $\prec$ on $\lbrace m_1, \ldots, m_r \rbrace$. 
After relabelling we may assume that $m_1 \prec m_2 \prec \cdots \prec m_r$. Let $\Delta$ be a simplicial complex on the vertex set $\lbrace 1,\ldots,r \rbrace$. 
By abuse of notation, we will say $\Delta$ is a simplicial complex on vertex set $\lbrace m_1, \ldots, m_r \rbrace$.\\
Assign a multidegree $m_J$ to each simplex $J \in \Delta$ by defining 
\[
m_J = \lcm \lbrace m_j \mid j \in J \rbrace.
\]
Define a chain complex $F_{\Delta}$ associated to $\Delta$ as follows. 
Let $F_n$ be the free $S$-module on generators $u_J$ with $\vert J \vert = n$.
For $J = \lbrace j_1 \prec \cdots \prec j_n \rbrace$, put $J^i = \lbrace j_1 \prec \cdots \prec \widehat{j_i} \prec \cdots \prec j_n \rbrace$. 
The differential $d\colon F_n \to F_{n-1}$ is defined, for $J \in \Delta$, by 
\[
 d(u_J) = \sum_{i=1}^{\vert J \vert} (-1)^{i+1} \frac{m_J}{m_{J^i}} u_{J^i}.
\]
\begin{example}
Let $\Delta^r$ be the full $r$-simplex. Then $F_{\Delta^r}$ is the Taylor resolution of $R = S/I$. This also justifies the use of the same notation for both. 
\end{example}

In general, $F_{\Delta}$ need not be a resolution of $S / I$. However, we have the following theorem.

\begin{theorem}[\cite{bayerpeevasturmfels1998}, Lemma 2.2]
 Let $\Delta$ be a simplicial complex on vertex set $\lbrace m_1, \ldots, m_r \rbrace$ and define, for a multidegree $\mu$, a subcomplex 
 \[
  \Delta_{\mu} = \lbrace J \in \Delta \mid m_J \mbox{ divides } \mu \rbrace.
 \]
 Then $F_{\Delta}$ is a resolution of $R$ if and only if $\Delta_{\mu}$ is either acyclic or empty for all multidegrees $\mu$. 
\end{theorem}

A resolution $F$ is called a \emph{simplicial resolution} if $F = F_{\Delta}$ for some simplicial complex $\Delta$. 

\begin{remark}
\label{twonotations}
Note that if $\Delta' \subseteq \Delta$, then $F_{\Delta'}$ is a subcomplex of $F_{\Delta}$.
In particular, since each simplicial complex $\Delta$ is included in the full simplex on its vertex set, each simplicial resolution of $S/I$ is a subcomplex of the Taylor resolution of $S/I$. 
\end{remark}

In the rest of the paper we will restrict our attention to the following special type of simplicial resolution which is due to Novik \cite{novik2002}. 
Given an monomial ideal $I = (m_1, \ldots, m_r)$ we define the \emph{lcm-lattice} $L(I)$ to be the set of all $\lcm(m_{i_1},\ldots, m_{i_k})$ where $1 \leq i_1 \leq \cdots \leq i_k \leq r$ and $k=1,\ldots, r$. 
The set $L=L(I)$ admits a partial order given by divisibility. Then $L$ forms a lattice under $a \vee b = \lcm(a,b)$ and $a \wedge b = \gcd(a,b)$. The lattice $L$ has minimal element $\hat{0} = 1$ and maximal element $\hat{1} = \lcm(m_1, \ldots, m_r)$.
\begin{definition}
 A \emph{rooting map} on $L$ is a map $\pi\colon L \setminus \lbrace \hat{0} \rbrace \to \lbrace m_1, \ldots, m_r \rbrace$ such that
\begin{enumerate}
 \item for every $m \in L$, $\pi(m)$ divides $m$
 \item $\pi(m) = \pi(n)$ whenever $\pi(m)$ divides $n$ and $n$ divides $m$. 
\end{enumerate}
\end{definition}

Now, let $\pi$ be a rooting map and let $A \subseteq \lbrace m_1, \ldots, m_r \rbrace$ be non-empty. Define $\pi(A) = \pi(\lcm(A))$. 
A set $A$ is \emph{unbroken} if $\pi(A) \in A$ and $A$ is \emph{rooted} if every non-empty $B \subseteq A$ is unbroken. 
Let $RC(L,\pi)$ denote the set of all rooted sets with respect to $L$ and $\pi$. Then $RC(L,\pi)$ is easily seen to be a simplicial complex on vertex set $\lbrace m_1, \ldots, m_r \rbrace$ and we have the following result.
\begin{theorem}[\cite{novik2002}, Theorem 1]
 Let $I = (m_1, \ldots, m_r)$ be a monomial ideal and let $L$ denote its lcm-lattice. Suppose that $\pi$ is a rooting map on $L$. 
 Then the chain complex $F_{RC(L,\pi)}$ associated to the simplicial complex $RC(L,\pi)$ is a free resolution of $I$. 
\end{theorem}

An important special case of this construction is the Lyubeznik resolution:

\begin{definition}
\label{lyubeznikresolution}
Let $I = (m_1, \ldots, m_r)$ be a monomial ideal and pick some total order $\prec$ on the $m_i$. After relabelling we may assume that $m_1 \prec m_2 \prec \cdots \prec m_r$. 
Define 
\[
 \pi(A) = \min_{\prec}\lbrace m_i \mid m_i \mbox{ divides } \lcm(A) \rbrace.
\]
Then $\pi$ is easily seen to be a rooting map. The resolution associated $RC(L,\pi)$ is called the \emph{Lyubeznik resolution}. 
\end{definition}

In this paper we will only consider resolutions $F$ that are as small as possible in the sense that each $F_n$ has the minimal number of generators. 
More precisely, we have the following definition. 

\begin{definition}
Let $S/I$ be a monomial ring. A free resolution $F \to S/I$ is said to be \emph{minimal} if $d(F) \subseteq (x_1,\ldots,x_m)F$. 

\end{definition}

If the minimal free resolution of $S/I$ is a resolution associated to $RC(L,\pi)$ for some rooting map $\pi$, then $I$ (respectively $S/I$) is called a \emph{rooted ideal} (respectively a \emph{rooted ring}). 
Similarly, if the Lyubeznik resolution of $S/I$ is minimal then $I$ (respectively $S/I$) is called a \emph{Lyubeznik ideal} (respectively a \emph{Lyubeznik ring}). 

\begin{example}
Let $S = k[x,y,z]$ and let $I$ be the ideal generated by $m_1 = xy$, $m_2 = yz$ and $m_3 = xz$. Order the generators as $m_1 \prec m_2 \prec m_3$. 
Let $\pi$ be the rooting map of the Lyubeznik resolution as in Definition \ref{lyubeznikresolution}. Then the rooted sets are $m_1$, $m_2$, $m_3$, $m_1m_2$ and $m_1m_3$. So the Lyubeznik resolution is
\begin{center}
\begin{tikzpicture}[descr/.style={fill=white,inner sep=1.5pt}]
        \matrix (m) [
            matrix of math nodes,
            row sep=3em,
            column sep=5em,
            text height=1.5ex, text depth=0.25ex
        ]
        { S^2 & S^3 & S \\
         };

        \path[overlay,->, font=\small, >=latex]
                        (m-1-1) edge node[yshift=1.5ex] {$d_2$} (m-1-2) 
                        (m-1-2) edge node[yshift=1.5ex] {$d_1$} (m-1-3);
\end{tikzpicture}   
\end{center} 
where the differential is given by
\begin{center}
\begin{tikzpicture}[descr/.style={fill=white,inner sep=1.5pt}]
        \matrix (m) [
            matrix of math nodes,
            row sep=3em,
            column sep=5em,
            text height=1.5ex, text depth=0.25ex
        ]
        { d_1 = \begin{bmatrix}xy & yz & xz\end{bmatrix} & \text{and} & d_2 = \begin{bmatrix}
-z & -z \\ x & 0 \\ 0 & y \end{bmatrix}. \\
         };
\end{tikzpicture}   
\end{center} 
In particular, the resolution is minimal and so $I$ is a Lyubeznik ideal. 
\end{example}

We point out that the class of rooted rings is fairly general. It includes for example monomial ideals whose lcm lattice is a geometric lattice as well as matroid ideals of modular matroids \cite{novik2002}. The inclusion of Lyubeznik rings in rooted rings is strict since not every rooting map arises from a total order on the monomial generators as Example 4.1 of \cite{bjornerziegler1991} shows. Finally, not every monomial ring is rooted. Let $I$ be the ideal with monomial generators
\begin{center}
\begin{tikzpicture}[descr/.style={fill=white,inner sep=1.5pt}]
        \matrix (m) [
            matrix of math nodes,
            row sep=0.5em,
            column sep=3em,
            text height=1.5ex, text depth=0.25ex
        ]
        {x_1x_4x_5x_6 & x_2x_4x_5x_6 & x_3x_4x_5x_6 & x_2x_4x_5x_7 & x_3x_4x_5x_7 \\
         x_1x_3x_5x_7 & x_1x_2x_4x_7 & x_1x_4x_6x_7 & x_1x_5x_6x_7 & x_3x_4x_6x_7 \\
         & x_2x_5x_6x_7 & x_2x_3x_6x_7 & x_1x_2x_3x_7 \\
         };
\end{tikzpicture}   
\end{center} 
and let $F$ denote the minimal free resolution. As is shown in \cite{reinerwelker2001}, the matrices of the differential of $F$ cannot be chosen in $\{0, \pm 1 \}$ and consequently $F$ cannot be supported on \emph{any} simplicial complex and hence, in particular, not on a complex $RC(L,\pi)$ coming from a rooting map $\pi$.

\section{$A_{\infty}$-algebras}
In this section we will discuss some basic aspects of the theory of $A_{\infty}$-algebras. The notion was first introduced by Stasheff \cite{stasheff1963} in the context of algebraic topology. 
Since their introduction $A_{\infty}$-algebras have found applications in various branches of mathematics such as geometry \cite{getzlerjones1990}, algebra \cite{stasheff1992} and mathematical physics \cite{kontsevich1995}, \cite{mccleary1999}. 
Though the following section aims to be self-contained, a more extensive introduction can be found in \cite{keller2001}. The exposition below follows that of \cite{lupalmieriwuzhang2009}.\\
In what follows all signs are determined by the \emph{Koszul sign convention}
\begin{equation}
\label{koszulsignconvention}
(f \otimes g) (x \otimes y) = (-1)^{\vert g \vert \cdot \vert x \vert} fx \otimes gy. 
\end{equation}

\begin{definition}
Let $R$ be a commutative ring and $A = \oplus A_n$ a $\Z$-graded free $R$-module. An $A_{\infty}$-algebra structure on $A$ consists of maps $\mu_n\colon A^{\otimes n} \to A$ for each $n \geq 1$ of degree $n-2$ satisfying the \emph{Stasheff identities}
\begin{equation}
\label{stasheffidentities}
 \sum (-1)^{r+st} \mu_u(1^{\otimes r} \otimes \mu_s \otimes 1^{\otimes t}) = 0
\end{equation}
where the sum runs over all decompositions $n=r+s+t$ with $r,t \geq 0$, $s \geq 1$ and $u=r+t+1$. 
\end{definition}
Observe that when applying (\ref{stasheffidentities}) to an element additional signs appear because of the Koszul sign convention (\ref{koszulsignconvention}). In the special case when $\mu_3=0$, it follows that $\mu_2$ is strictly associative and so $A$ is a differential graded algebra with differential $\mu_1$ and multiplication $\mu_2$. An $A_{\infty}$-algebra $A$ is called \emph{strictly unital} if there exists an element $1 \in A$ that is a unit for $\mu_2$ and such that for all $n \neq 2$
$$\mu_n(a_1 \otimes \cdots \otimes a_n) = 0$$
whenever $a_i=1$ for some $i$. 

The notion of a morphism between $A_{\infty}$-algebras will also be needed. 
\begin{definition}
 Let $(A, \mu_n)$ and $(B,\overline{\mu}_n)$ be $A_{\infty}$-algebras. A \emph{morphism} of $A_{\infty}$-algebras (or $A_{\infty}$\emph{-morphism}) $f\colon A \to B$ is a family of linear maps 
 \[
  f_n\colon A^{\otimes n} \to B
 \]
of degree $n-1$ satisfying the \emph{Stasheff morphism identities}
\begin{equation}
 \label{stasheffmorphismidentities}
 \sum (-1)^{r+st}f_u(1^{\otimes r} \otimes \mu_s \otimes 1^{\otimes t}) = \sum (-1)^w \overline{\mu}_q(f_{i_1} \otimes f_{i_2} \otimes \cdots \otimes f_{i_q})
\end{equation}
for every $n \geq 1$. The first sum runs over all decompositions $n=r+s+t$ with $s \geq 1$ and $r,t \geq 0$ where $u=r+t+1$. The second sum runs over all $1 \leq q \leq n$ and all decompositions $n = i_1 + i_2 + \cdots + i_q$ with all $i_s \geq 1$. 
The sign on the right-hand side of (\ref{stasheffmorphismidentities}) is given by
\[
 w = \sum_{p=1}^{q-1}(q-p)(i_p-1).
\]
If $A$ and $B$ are strictly unital, an $A_{\infty}$-morphism is also required to satisfy $f_1(1) = 1$ and
\[
 f_n(a_1 \otimes \cdots \otimes a_n) = 0
\]
if $n \geq 2$ and $a_i = 1$ for some $i$. 
\end{definition}
A morphism $f$ is called a \emph{quasi-isomorphism} if $f_1$ is a quasi-isomorphism in the usual sense.

Let $A$ be an $A_{\infty}$-algebra. Then its homology $HA$ is an associative algebra. A crucial result relating the $A_{\infty}$-algebra $A$ and its homology algebra $HA$ is the \emph{homotopy transfer theorem}. 
\begin{theorem}[Homotopy Transfer Theorem, \cite{kadeishvili1980}, see also \cite{merkulov1999}] 
\label{homotopytransfertheorem}
Let $(A, \mu_n)$ be an $A_{\infty}$-algebra over a field $R$ and let $HA$ be its homology algebra. 
There exists an $A_{\infty}$-algebra structure $\mu'_n$ on $HA$ such that
\begin{enumerate}
 \item $\mu'_1 = 0$, $\mu'_2 = H(\mu_2)$ and the higher $\mu'_n$ are determined by $\mu_n$
 \item there exists an $A_{\infty}$-quasi-isomorphism $HA \to A$ lifting the identity morphism of $HA$.
\end{enumerate}
Moreover, this $A_{\infty}$-structure is unique up to isomorphism of $A_{\infty}$-algebras.
\end{theorem}

An explicit way of contructing $A_{\infty}$-structures on the homology of a dga is due to Merkulov \cite{merkulov1999} and will be discussed next. 
\begin{definition}
Let $A$ be a chain complex and $B \subseteq A$ a subcomplex. A \emph{transfer diagram} is a diagram of the form
\begin{equation}
\label{transferdiagram}
\begin{tikzpicture}[baseline=(current  bounding  box.center)]
\matrix(m)[matrix of math nodes,
row sep=3em, column sep=2.8em,
text height=1.5ex, text depth=0.25ex]
{B &A\\};
\path[->]
(m-1-1) edge [bend left=35] node[yshift=1.5ex] {$i$} (m-1-2)
(m-1-2) edge [bend left=35] node[yshift=-1.5ex] {$p$} (m-1-1)
(m-1-2) edge [loop right, in=35,out=-35,looseness=5, min distance=10mm] node {$\phi$} (m-1-2)
;
\end{tikzpicture}
\end{equation}
where $pi = 1_B$ and $ip - 1 = d \phi + \phi d$. 
\end{definition}

Some authors use the term strong deformation retract for what we call a transfer diagram. Let $(A,d)$ be a dga and let $B$ be a subcomplex of $A$ such that there exists a transfer diagram as in \eqref{transferdiagram} Let $\cdot$ denote the product on $A$. Define linear maps $\lambda_n\colon A^{\otimes n} \to A$ as follows. First, put $\lambda_2(a_1,a_2) = a_1 \cdot a_2$ and we set
\begin{equation}
\label{merkulovlambda}
\lambda_n = \sum_{\substack{s+t = n \\ s,t \geq 1}} (-1)^{s+1} \lambda_2 (\phi \lambda_s, \phi \lambda_t)
\end{equation}
Now, define a second series of maps $\mu_n\colon B^{\otimes n} \to B$ by setting $\mu_1 = d$ and, for $n \geq 2$, 
\begin{equation}
\label{merkulovmun}
\mu_n = p \circ \lambda_n \circ i^{\otimes n}.
\end{equation}
The following theorem will be crucial in the remainder of the paper.
\begin{theorem}[\cite{merkulov1999}, Theorem 3.4]
\label{merkulovtheorem}
Let $(A,d)$ be a dga and $B$ a subcomplex of $A$ such that there exists a transfer diagram of the form \eqref{transferdiagram}. Then the maps $\mu_n$ defined in (\ref{merkulovmun}) give the structure of an $A_{\infty}$-algebra on $B$.  
\end{theorem}

\section{$A_{\infty}$-resolutions and the Golod property}
Let $R$ be a monomial ring. Recall that $R$ is called \emph{Golod} if there is an equality of power series
\begin{equation}
\label{goloddefinition}
 P(R) = \frac{(1+t)^m}{1-t(\sum_{j=0}^{\infty}\dim \tor^S_j(R,k) t^j -1)}
\end{equation}
The Golod property admits an equivalent description in terms of Massey products which will be defined next.

\begin{definition}
Let $(A,d)$ be a differential graded algebra. If $a \in A$, we write $\bar{a}$ for $(-1)^{\text{deg}(a) +1}a$. \\
Let $\alpha_1,\alpha_2 \in HA$. The length $2$ \textit{Massey product} $\langle \alpha_1, \alpha_2 \rangle$ is defined to be the product $\alpha_1 \alpha_2$ in the homology algebra $HA$. \\
Let $\alpha_1, \ldots, \alpha_n \in HA$ be homology classes with the property that each length $j-i+1$ Massey product $\langle \alpha_i, \ldots, \alpha_j \rangle$ is defined and contains zero for $i<j$ and $j-i < n-1$. A \emph{defining system} $\{ a_{ij} \}$ consists of
\begin{enumerate}
\item For $i=1,\ldots,n$, representing cycles $a_{i-1,i}$ of the homology class $\alpha_i$.
\item For $j > i+1$, elements $a_{ij}$ such that 
$$da_{ij} = \sum_{i<k<j} \bar{a}_{ik}a_{kj}.$$ 
\end{enumerate}
Note that the existence is guaranteed by the condition that $\langle \alpha_i, \ldots, \alpha_j \rangle$ is defined and contains zero for $i<j$ and $j-i < n-1$.
The length $n$ \textit{Massey product}$\langle \alpha_1, \ldots, \alpha_n \rangle$  is defined as the set
$$ \langle \alpha_1, \ldots, \alpha_n \rangle  = \{ [\sum_{0<i<n} \bar{a}_{0i} a_{in}]\mid \{ a_{ij} \} \mbox{ is a defining system } \} \subseteq H^{s+2-n}$$
where $s = \sum_{i=1}^n \deg \alpha_i$. 
\end{definition}

A Massey product $\langle \alpha_1, \ldots, \alpha_n \rangle$ is said to be \emph{trivial} if it contains zero. The \emph{Koszul homology} of a monomial ring $R$ is $H(R) = \tor^S(R,k)$. The Golod property and Massey products are related by the following theorem. 

\begin{theorem}[\cite{golod1978}, see also Section 4.2 of \cite{gulliksenlevin1969}]
\label{golodiffmasseytrivial}
Let $R$ be a monomial ring. Then $R$ is Golod if and only if all Massey products on the Koszul homology $\tor^S(R,k)$ are trivial.
\end{theorem}

Following \cite{katthan2016}, we will say that a dga $A$ satisfies \emph{condition} $(B_r)$ if all $k$-ary Massey products are defined and contain only zero for all $k \leq r$. Recall the following lemma.

\begin{lemma}[\cite{may1969}, Proposition 2.3]
\label{brimpliesuniquemassey}
 Let $A$ be a dg algebra satisfying $(B_{r-1})$. Then $\langle a_1, \ldots, a_r \rangle$ is defined and contains only one element for any choice $a_1,\ldots, a_r \in H(A)$. 
\end{lemma}

Let $R$ be a monomial ring and let $K_S$ be the Koszul resolution of the base field $k$ over $S$. The \emph{Koszul dga} $K_R$ of $R$ is defined as $K_R = R \otimes_S K_S$. The Koszul dga and the Taylor resolution are related by a zig-zag of dga quasi-isomorphisms  

\begin{center}
\begin{tikzpicture}[baseline=(current  bounding  box.center)]
\matrix(m)[matrix of math nodes,
row sep=3em, column sep=3em,
text height=1.5ex, text depth=0.25ex]
{T \otimes_S k &T \otimes_S K_S & R \otimes_SK_S = K_R\\};
\path[->]
(m-1-2) edge (m-1-1)
(m-1-2) edge (m-1-3)
;
\end{tikzpicture}
\end{center}
Consequently, Massey products on $\tor^S(R,k)$ can be computed using either $K_R$ or $T \otimes_S k$. Again following \cite{katthan2016}, we say that a monomial ring $R$ satisfies $(B_r)$ if the dga $K_R$ of $R$ satisfies $(B_r)$.

\begin{lemma}
 Let $R$ be a monomial ring. Then $R$ is Golod if and only if $R$ satisfies condition $(B_r)$ for all $r \in \N$. 
\end{lemma}
\begin{proof}
 It is clear that if $R$ satisfies condition $(B_r)$ for every $r$ then $R$ is Golod. Conversely, suppose that $R$ is Golod.
 We proceed by induction on $r$. The case $r=2$ is trivial. So assume $R$ satisfies $(B_{r-1})$. 
 By Lemma \ref{brimpliesuniquemassey}, the Massey product $\langle a_1, \ldots, a_r \rangle$ is defined and contains only one element for any choice $a_1,\ldots, a_r \in \tor^S(R,k)$.
 Since $R$ is Golod, it follows by Theorem \ref{golodiffmasseytrivial} that this element must be zero and so $R$ satisfies $(B_r)$.
\end{proof}

In general it is very hard to study Massey products directly. However, $A_{\infty}$-algebras provide a systematic way of studying Massey products in view of the following theorem.

\begin{theorem}[\cite{lupalmieriwuzhang2009}, Theorem 3.1]
\label{ainfinitymasseyproductsaremasseyproducts}
 Let $A$ be a differential graded algebra. Up to a sign, the higher $A_{\infty}$-multiplications $\mu'_n$ on $HA$ from Theorem \ref{homotopytransfertheorem} give Massey products. 
 That is to say, if $\alpha_1, \ldots, \alpha_n \in HA$ are homology classes such that the Massey product $\langle \alpha_1, \ldots, \alpha_n \rangle$ is defined then
$$\pm \mu'_n(\alpha_1 \otimes \cdots \otimes \alpha_n) \in \langle \alpha_1, \ldots, \alpha_n \rangle.$$
\end{theorem}

A map of $S$-modules $f\colon M \to N$ is said to be minimal if $f \otimes 1 \colon M \otimes_S k \to N \otimes_S k$ is zero. It is readily verified that $f$ is minimal if and only if $f$ maps into $(x_1, \ldots, x_m)N$.  
Using Theorem \ref{ainfinitymasseyproductsaremasseyproducts}, we can describe under what conditions the Massey products on $\tor^S(R,k)$ vanish. 

\begin{theorem}
 Let $R = S/I$ be a monomial ring with minimal free resolution $F$. Let $r \in \N$ and let $\mu_n$ be an $A_{\infty}$-structure on $F$ such that $F \otimes_S k$ and $K_R$ are quasi-isomorphic as $A_{\infty}$-algebras. Then $R$ satisfies $(B_r)$ if and only if $\mu_k$ is minimal for all $k \leq r$. 
\end{theorem}
\begin{proof}
Since $\mu_n$ is an $A_{\infty}$-structure on $F$, it follows that $\mu_n \otimes 1$ is an $A_{\infty}$-structure on $F \otimes_S k$. Now, assume $\mu_n$ is minimal for all $k \leq r$. 
Since $\tor^S(R,k)$ is the homology of the $A_{\infty}$-algebra $F \otimes_S k$ the homotopy transfer theorem (Theorem \ref{homotopytransfertheorem}) implies that $\tor^S(R,k)$ inherits an $A_{\infty}$-structure $\mu'_n$. 
Since $F$ is minimal, $\tor^S(R,k)$ is isomorphic to $F \otimes_S k$ and we can take $\mu'_n = \mu_n \otimes 1$. Let $k \leq r$ and let $\alpha_1, \ldots, \alpha_k \in \tor^S(R,k)$ be such that the Massey product $\langle \alpha_1,\ldots, \alpha_k \rangle$ is defined. By Theorem \ref{ainfinitymasseyproductsaremasseyproducts} we have
$$\pm (\mu_k \otimes 1)(\alpha_1, \ldots, \alpha_k) \in \langle \alpha_1, \ldots, \alpha_k \rangle.$$
Since $\mu_k$ is minimal, we have $(\mu_k \otimes 1)(\alpha_1, \ldots, \alpha_k)=0$. Therefore, $\langle \alpha_1,\ldots, \alpha_k \rangle$ is trivial and so $R$ satisfies $(B_r)$. \\
Conversely, assume that $R$ satisfies $(B_r)$. We need to show that $\mu_k$ is minimal for all $k \leq r$. 
For $k=2$, we have $(\mu_2 \otimes 1)(a_1, a_2) = a_1a_2$ but the product on $\tor^S(R,k)$ is zero as $R$ satisfies $(B_r)$.
Now, let $3 \leq k \leq r$. Since $R$ satisfies $(B_k)$, for all $a_1, \ldots, a_k$ the Massey product $\langle a_1, \ldots, a_k \rangle$ is defined and contains only zero.
Since $(\mu_k \otimes 1)(a_1, \ldots, a_k) \in \langle a_1, \ldots, a_k \rangle$ we have $(\mu_k \otimes 1)(a_1, \ldots, a_k) = 0$ for all $a_1, \ldots, a_k$. 
Consequently, $\mu_k$ is minimal as required. 
\end{proof}

\begin{corollary}
\label{munminimalimpliesgolod}
Let $R = S/I$ be a monomial ring with minimal free resolution $F$. Let $\mu_n$ be an $A_{\infty}$-structure on $F$ such that $F \otimes_S k$ and $K_R$ are quasi-isomorphic as $A_{\infty}$-algebras. Then $R$ is Golod if and only if $\mu_n$ is minimal for all $n \geq 1$.   
\end{corollary}

Corollary \ref{munminimalimpliesgolod} was first proved in \cite{burke2015} using different methods. 

The following immediate corollary to Theorem \ref{munminimalimpliesgolod} is well-known, see for example Proposition 5.2.4(4) of \cite{avramov2010} where it is proved using different methods. 

\begin{corollary}[\cite{avramov2010}, Proposition 5.2.4(4)]
\label{dgagolodproducttrivial}
Let $R = S/I$ be a monomial ring with minimal free resolution $F$. If $F$ admits the structure of a dga, then $R$ is Golod if and only if the product on $\tor^S(R,k)$ vanishes.
\end{corollary}

\section{Homotopy transfer on the Taylor resolution}
Theorem \ref{munminimalimpliesgolod} implies that monomial rings with minimal dg algebra resolution are Golod if and only if the product on $\tor^S(S/I,k)$ vanishes. 
However, there exists monomial rings whose minimal resolution does not admit the structure of a dg algebra \cite{avramov1981}. On the other hand, every free resolution of a monomial ring $S/I$ admits an $A_{\infty}$-structure \cite{burke2015}.

In general, it is not clear how to obtain an explicit description of such an $A_{\infty}$-structure. Instead of considering general $A_{\infty}$-structures on resolutions, we will consider only those that arise as a deformation of the dg algebra structure on the Taylor resolution. To make this idea precise we will use rooting maps to construct transfer diagrams on the Taylor resolution.
In that case Theorem \ref{merkulovtheorem} tells us how to construct an $A_{\infty}$-structure to which we may apply Theorem \ref{munminimalimpliesgolod}. \\

Let $\pi$ be a rooting map and let $F$ be the free resolution of $S/I$ associated to $RC(L,\pi)$. Recall that $F_n$ is the free $S$-module on $u_J$ where $J \in RC(L, \pi)$ with $\vert J \vert =n$. 
The remainder of this section is devoted to computing an explicit $A_{\infty}$-algebra structure on $F$.
Let $T$ will denote the Taylor resolution of $S/I$. We will write $d$ for the differential of $F$ whereas $\partial$ will be reserved for the ``simplicial'' differential, i.e.
$$\partial u_J = \sum_{i=1}^{\vert J \vert} (-1)^{i+1} u_{J^i}$$
on a basis set $u_J$ of $F$.
If $u_J$ is a basis set of $F$ we define $[u_J] = \frac{1}{m_J}u_J$. Let $u_{J_1}, \ldots, u_{J_n}$ be rooted sets and $\alpha_1, \ldots, \alpha_n \in S$. 
Then for $u = \sum \alpha_k u_{J_k}$, we set $[u] = \sum \frac{\alpha_k}{m_{J_k}} u_{J_k}$.
The following lemma will be used extensively.

\begin{lemma}
For any basis set $u_J$ of $F$ we have $d[u_J] = [\partial u_J]$.
\end{lemma}
\begin{proof}
We have
\begin{equation*}
\begin{split}
d[u_J] &= \frac{1}{m_J} du_J =\frac{1}{m_J} \sum_{i=1}^{\vert J \vert} (-1)^{i+1} \frac{m_J}{m_{J^i}} u_{J^i}\\
&= \sum_{i=1}^{\vert J \vert} (-1)^{i+1} \frac{1}{m_{J^i}} u_{J^i} = \sum_{i=1}^{\vert J \vert} (-1)^{i+1} [u_{J^i}] \\
&=[\partial u_J].
\end{split}
\end{equation*}
\end{proof}
Let $\pi$ be a rooting map. For $u_J \in T$, define $\pi(u_J) = u_i$ if $\pi(m_J) = m_i$.
Define a map $p'\colon T \to F$ as follows. Let $u \in T$ and write $u = u_{i_1} \cdots u_{i_k}$. 
For $q=1,\ldots,k$ define $I_q = \lbrace i_1,\ldots,i_q \rbrace$. For a permutation $\sigma \in S_k$, put $\sigma I_q = \{ i_{\sigma(1)}, \ldots ,i_{\sigma(q)} \}$. We define
\begin{equation}
\label{definitionp'}
 p'(u) = \sum_{\sigma \in S_k} \sgn(\sigma) \pi(u_{\sigma I_1})\pi(u_{\sigma I_2}) \cdots \pi(u_{\sigma I_k}).  
\end{equation}
Geometrically, the map $p'$ can be thought of as similar to the barycentric subdivision of a simplex. For example, if $u_{i_1,i_2} \in T$ and we think of $\pi(u_{i_1,i_2})$ as its barycenter then $p'$ replaces $u_{i_1,i_2}$ by its barycentric subdivision
$$p'(u_{i_1,i_2}) = u_{i_2}\pi(u_{i_1,i_2}) - u_{i_1}\pi(u_{i_1,i_2}).$$
In the same way, given $u_{i_1,i_2,i_3} \in T$ the right hand terms in
$$p(u_{i_1,i_2,i_3}) = \sum_{\sigma \in S_3} \sgn(\sigma) \pi(u_{i_{\sigma(1)}})\pi(u_{i_{\sigma(1)},i_{\sigma(2)}})\pi(u_{i_{\sigma(1)},i_{\sigma(2)}, i_{\sigma(3)}})$$
are precisely the six constituent triangles in the barycentric subdivision of a $2$-simplex. Before proceeding, we need to verify that $\im(p') \subseteq F$. Let $\sigma \in S_k$, we need to show that
$$\pi(u_{\sigma I_1})\pi(u_{\sigma I_2}) \cdots \pi(u_{\sigma I_k})$$
is rooted. Since $u_{\sigma I_1} \subseteq u_{\sigma I_2} \subseteq \cdots \subseteq u_{\sigma I_k}$, it follows that for all $j_1, \ldots, j_k$ we have
$$\pi ( \pi(u_{I_{j_1}}), \pi(u_{I_{j_2}}),\ldots,\pi(u_{I_{j_k}})) = \pi(u_{I_{j_k}}).$$
Therefore, $\pi(u_{\sigma I_1})\pi(u_{\sigma I_2}) \cdots \pi(u_{\sigma I_k})$
is rooted and so $\im(p') \subseteq F$. 

\begin{lemma}
\label{pprimechainmappartial}
The map $p'$ is a chain map with respect to the differential $\partial$. 
\end{lemma}
\begin{proof}
It is sufficient to prove the result for basis elements $u_I \in T$. Write $I = \lbrace i_1, \ldots, i_k \rbrace$.
We first show that
$$\partial p'(u_I) = \sum_{\sigma \in S_k} (-1)^{k+1} \sgn(\sigma) \pi(u_{\sigma I_1}) \cdots \pi(u_{\sigma I_{k-1}}).$$
We have
\begin{equation*}
\begin{split}
\partial p'(u_I) &= \sum_{\sigma \in S_k} \sgn(\sigma)  \partial \big( \pi(u_{\sigma I_1}) \cdots \pi(u_{\sigma I_k})\big) \\
&= \sum_{\sigma \in S_k} \sum_{j=1}^k (-1)^{j+1} \sgn(\sigma)  \pi(u_{\sigma I_1}) \cdots  \widehat{\pi(u_{\sigma I_j})} \cdots \pi(u_{\sigma I_k}).
\end{split}
\end{equation*}
Now, fix some $j<k$ and let $\tau_j$ be the transposition $(\sigma(j), \sigma(j+1))$. Then the summands indexed by $\sigma$ and $\tau_j\sigma$ cancel. 
Indeed, if $q < j$ then $\tau_j$ acts as the identity on $\sigma I_q$ and so $u_{\sigma I_q} = u_{\tau_j \sigma I_q}$. On the other hand, if $q  \geq j+1$ then the underlying sets of $\sigma I_q$ and $\tau_j \sigma I_q$ are the same. Since $\pi(u_J)$ depends only on the set $J$ and not on the ordering we have 
$$\pi(u_{\sigma I_q}) = \pi(u_{\tau_j \sigma I_q})$$
and so the summands indexed by $\sigma$ and $\tau_j\sigma$ cancel. Note that since the map $\sigma \to \tau_j \sigma$ is an involution these permutations cancel in pairs. Therefore, we obtain
$$\partial p'(u_I) = \sum_{\sigma \in S_k} (-1)^{k+1} \sgn(\sigma) \pi(u_{\sigma I_1}) \cdots \pi(u_{\sigma I_{k-1}}).$$
For $\sigma \in S_k$, write 
$$G_{\sigma} = \pi(u_{\sigma I_1}) \cdots \pi(u_{\sigma I_{k-1}})$$
and so
\begin{equation}
\label{gsigma}
\partial p'(u_I) = \sum_{\sigma \in S_k} (-1)^{k+1} \sgn(\sigma) G_{\sigma}.
\end{equation}
Next, we compute $p' \partial (u_I)$. For $j \in \{ 1,\ldots, k \}$ and $\sigma \in S_{k-1}$, set $I_q(j) = I_q \setminus \{ j \}$ and
$$F_{\sigma, j} = \pi( u_{\sigma I_1(j)} ) \cdots \pi( u_{\sigma I_{j-1}(j)} )\pi( u_{\sigma I_{j+1}(j)} ) \cdots \pi( u_{\sigma I_k(j)} ).$$
Then
\begin{equation}
\label{fsigmaj}p' \partial u = \sum_{j=1}^k (-1)^{j+1} p'(u_{I_k(j)}) = \sum_{j=1}^k \sum_{\sigma \in S_{k-1}} (-1)^{j+1} \sgn (\sigma) F_{\sigma,j}.
\end{equation}
Given $j \in \{ 1,\ldots, k \}$, we can embed $S_{k-1}$ into $S_k$ by fixing $j$. Therefore, we have
$$p' \partial u = \sum_{j=1}^k \sum_{\sigma \in S_{k-1}} (-1)^{j+1} \sgn (\sigma) F_{\sigma,j} = \sum_{j=1}^k \sum_{\substack{\sigma \in S_k \\ \sigma(j) = j}}(-1)^{j+1} \sgn (\sigma) F_{\sigma,j}.$$
Now, fix $j \in \{ 1,\ldots, k \}$ and fix $\sigma \in S_k$ such that $\sigma(j) = j$. Define $\rho$ to be the cycle $(j \cdots k)$ and let $\tau = \sigma \rho$. Then we have $G_{\tau} = F_{\sigma,j}$ and
$$(-1)^{k+1} \sgn (\tau) G_{\tau} = (-1)^{2k+j+1} G_{\sigma \rho} = (-1)^{j+1} \sgn (\sigma) F_{\sigma,j}.$$
Since both sums in \eqref{gsigma} and \eqref{fsigmaj} have $k!$ terms, it follows that they are equal.
\end{proof}

Let $i\colon F \to T$ denote the inclusion.

\begin{lemma}
\label{piuip=ipu}
For all $u \in T$, we have
$$\pi(u) ip'\partial u = ip' u.$$
\end{lemma}
\begin{proof}
It is sufficient to prove the result for basis elements $u_I \in T$. Write $I = \lbrace i_1, \ldots, i_k \rbrace$. As in the proof of Lemma \ref{pprimechainmappartial}, we have
$$\partial p'(u_I) = \sum_{\sigma \in S_k} (-1)^{k+1} \sgn(\sigma) \pi(u_{\sigma I_1}) \cdots \pi(u_{\sigma I_{k-1}}).$$
Since $p'$ is a chain map by Lemma \ref{pprimechainmappartial}, we have
\begin{equation*}
\begin{split}
\pi(u_I) ip' \partial u_I &= \pi(u_I) \partial ip'(u_I) \\
&= \pi(u_I) \sum_{\sigma \in S_k} (-1)^{k+1} \sgn(\sigma) \pi(u_{\sigma I_1}) \cdots \pi(u_{\sigma I_{k-1}}) \\
&= \sum_{\sigma \in S_k} (-1)^{k+1+k-1} \sgn(\sigma) \pi(u_{\sigma I_1}) \cdots \pi(u_{\sigma I_{k-1}}) \pi(u_I)\\
&= \sum_{\sigma \in S_k} \sgn(\sigma) \pi(u_{\sigma I_1}) \cdots \pi(u_{\sigma I_{k-1}}) \pi(u_\sigma I_k)\\
&= ip'(u_I)
\end{split}
\end{equation*}
where we have used that $\pi(u_I) = \pi(u_{I_k}) = \pi(u_{\sigma I_k})$. 
\end{proof}

\begin{lemma}
\label{ip'chainhomotopyto1}
 The composition $ip'$ is chain homotopic to $1_T$ as chain maps $(T,\partial) \to (T,\partial)$.
\end{lemma}
\begin{proof}
 Define $\phi'\colon T \to T$ by induction as follows. Set $\phi'_0 = \phi'_1=0$ and 
$$\phi'_2(u_{i_1}u_{i_2}) = \pi(u_{i_1,i_2}) u_{i_1} u_{i_2}.$$
For $k>2$, write $u = u_{i_1} \cdots u_{i_k}$ and define
$$\phi'_k(u) = \pi(u)\big( u - \phi'_{k-1}(\partial u) \big).$$
We need to show that $1_T - ip' = \partial \phi' + \phi' \partial$. We proceed by induction on $k$.
If $k=1$, there is nothing to prove. If $k=2$, we have
\begin{equation*}
\partial \phi_2(u_{i_1}u_{i_2}) = \partial(\pi(u_{i_1,i_2}) u_{i_1} u_{i_2}) = u_{i_1}u_{i_2} - \pi(u_{i_1,i_2})u_{i_2} + \pi(u_{i_1,i_2}) u_{i_1} = (1_F - ip')(u_{i_1}u_{i_2}). 
\end{equation*}
Now, let $k>2$. Using Lemma \ref{piuip=ipu}, we get
\begin{equation*}
\begin{split}
\partial \phi'_k(u) &= u-\phi'_{k-1} \partial u - \pi(u)(\partial u - \partial \phi'_{k-1} \partial u) \\
&= u-\phi'_{k-1} \partial u - \pi(u) \big( \partial u - \partial u + ip' \partial u + \phi_{k-2} \partial^2 u \big) \\
&= u - \phi_{k-1} \partial u - \pi(u) ip'\partial u \\
&= u-ip'u - \phi_{k-1} \partial u
\end{split}
\end{equation*}
which finishes the proof. 
\end{proof}
Define a map $p\colon T \to F$ as follows. For $u_J \in T$, define
\begin{equation}
  \label{definitionp}
  p(u_J) = m_J [p'(u_J)]
\end{equation}
 where $p'$ is the map from (\ref{definitionp'}). Then we have the following theorem. 
\begin{theorem}
\label{rootedresolutionsaresdrtaylor}
 Let $\pi$ be a rooting map for a monomial ideal $I$ and let $F$ be the resolution of $S/I$ associated to $\pi$. Then there exists a transfer diagram
 \begin{center}
\begin{tikzpicture}
\matrix(m)[matrix of math nodes,
row sep=3em, column sep=2.8em,
text height=1.5ex, text depth=0.25ex]
{F &T\\};
\path[->]
(m-1-1) edge [bend left=35] node[yshift=1.5ex] {$i$} (m-1-2)
(m-1-2) edge [bend left=35] node[yshift=-1.5ex] {$p$} (m-1-1)
(m-1-2) edge [loop right, in=35,out=-35,looseness=5, min distance=10mm] node {$\phi$} (m-1-2)
;
\end{tikzpicture}
\end{center}
where $i\colon F \to T$ is the inclusion and $p\colon T \to F$ is the map from (\ref{definitionp}). 
\end{theorem}
\begin{proof}
Let $u_J \in T$ and define $\phi$ by $\phi(u_J) = m_J [\phi'(u_J)]$.
Then, using Lemma \ref{ip'chainhomotopyto1}, we have
$$ d \phi(u_J) =m_J d[\phi'(u_J)] = m_J [ \partial \phi'(u_J) ] = m_J [ u_J - ip'u_J - \phi' \partial u_J] = u_J - ipu_J - \phi du_J$$
and so $1_T$ and $ip$ are homotopic. On the other hand, we clearly have $pi = 1_F$ which finishes the proof. 
\end{proof}

\section{The Golod property for rooted rings}

Let $R=S/I$ be a rooted ring with rooting map $\pi$ and minimal free resolution $F$. The purpose of this section is to provide necessary and sufficient conditions for $R$ being Golod. Following \cite{jollenbeck2006}, we have the following definition.

\begin{definition}
 Let $R=S/I$ be a monomial ring and write $I=(m_1,\ldots,m_r)$. We say that $R$ satisfies the \emph{gcd condition} if for all generators $m_i$ and $m_j$ with $\gcd(m_i,m_j)=1$ there exists a $m_k \neq m_i,m_j$ such that $m_k$ divides $\lcm(m_i,m_j)$. 
\end{definition}

We have the following lemma where we write $\pi(m_i,m_j)$ for $\pi(\{ m_i,m_j\})$. 

\begin{lemma}
\label{gcdpigcd}
Let $R=S/I$ be a rooted ring with rooting map $\pi$. Write $I=(m_1,\ldots,m_r)$. Then $R$ satisfies the gcd condition if and only if $\pi(m_i,m_j) \neq m_i,m_j$ whenever $\gcd(m_i,m_j)=1$.
\end{lemma}
\begin{proof}
First, assume that $\pi(m_i,m_j) \neq m_i,m_j$ whenever $\gcd(m_i,m_j)=1$. Since $\pi(m_i,m_j)$ divides $\lcm(m_i,m_j)$, we can take $m_k = \pi(m_i,m_j)$ and so $R$ satisfies the gcd condition. 

Conversely, suppose that $R$ satisfies the gcd condition and take $m_i$ and $m_j$ with $\gcd(m_i,m_j)=1$. For contradiction, assume that $\pi(m_i,m_j) = m_i$. By the gcd condition, there exists some $m_k \neq m_i,m_j$ such that $m_k$ divides $\lcm(m_i,m_j)$. We claim that the set $\{ m_i,m_j, \pi(m_j,m_k) \}$ is rooted. To prove this, we need to verify that every subset is unbroken. Since $\pi(m_i,m_j) = m_i$, it follows immediately that $\{ m_i, m_j \}$ is unbroken. For $\{ m_j, \pi(m_j,m_k) \}$, note that
$$\pi(m_j,m_k) \vert \lcm(m_j, \pi(m_j,m_k)) \vert \lcm(m_j,m_k)$$
and so $\pi(m_j, \pi(m_j,m_k)) = \pi(m_j,m_k)$ as $\pi$ is a rooting map. Therefore, $\{ m_j, \pi(m_j,m_k) \}$ is unbroken. Next, consider $\{ m_i, \pi(m_j,m_k) \}$. Since $\pi(m_i,m_j) = m_i$, we have
$$\pi(m_i,m_j) \vert \lcm(m_i,\pi(m_j,m_k)) \vert \lcm(m_i,m_j)$$
and so $\pi(m_i,\pi(m_j,m_k)) = \pi(m_i,m_j)=m_i$. Consequently, $\{ m_i, \pi(m_j,m_k) \}$ is unbroken. Similarly, we have that $\{ m_i,m_j, \pi(m_j,m_k) \}$ is unbroken as
$$\pi(m_i,m_j) \vert \lcm(m_i,m_j,\pi(m_j,m_k)) \vert \lcm(m_i,m_j)$$
and thus $\pi(m_i,m_j,\pi(m_j,m_k)) = \pi(m_i,m_j)=m_i$. Therefore, $\{ m_i,m_j, \pi(m_j,m_k) \}$ is rooted as claimed. 

Let $u = u_iu_j\pi(u_j,u_k)$. Since $\pi(m_j,m_k)$ divides $\lcm(m_i,m_j)$, we have
\begin{equation*}
\begin{split}
du = \frac{\lcm(m_i,m_j)}{\lcm(m_j,\pi(m_j,m_k))}u_j\pi(u_j,u_k) -\frac{\lcm(m_i,m_j)}{\lcm(m_i,\pi(m_j,m_k))}u_i\pi(u_j,u_k) + u_iu_j.
\end{split}
\end{equation*}
Hence, $du \notin (x_1,\ldots,x_m)F$ which is a contradiction as $R$ is rooted. Therefore, $\pi(m_i,m_j) \neq m_i$. Swapping the roles of $i$ and $j$, we see that $\pi(m_i,m_j) \neq m_j$ which finishes the proof. 
\end{proof}

The following lemma is straightforward but included for completeness.

\begin{lemma}
 \label{plambda2minimalonnondisjoint}
 Let $u_I$ and $u_J$ be basis elements of $T$ with the property that $\gcd(m_I,m_J) \neq 1$. Then
$$p\lambda_2(u_I,u_J) \in (x_1,\ldots,x_m)F.$$
\end{lemma}
\begin{proof}
 Indeed, we have
\begin{equation*}
p\lambda_2(u_I\otimes u_J) = p(\frac{m_Im_J}{m_{I \cup J}}u_{I \cup J})
= \frac{m_Im_J}{m_{I \cup J}} p(u_{I \cup J}).
\end{equation*}
By assumption $\frac{m_Im_J}{m_{I \cup J}} \neq 1$ and so the result follows.
\end{proof}

\begin{lemma}
\label{pigcdimpliesgolod}
 Let $R$ be a rooted ring. If $R$ is gcd then $R$ is Golod
\end{lemma}
\begin{proof}
 Let $F$ be the minimal free resolution of $R$. Then by Theorem \ref{rootedresolutionsaresdrtaylor} there is a transfer diagram
  \begin{center}
\begin{tikzpicture}
\matrix(m)[matrix of math nodes,
row sep=3em, column sep=2.8em,
text height=1.5ex, text depth=0.25ex]
{F &T\\};
\path[->]
(m-1-1) edge [bend left=35] node[yshift=1.5ex] {$i$} (m-1-2)
(m-1-2) edge [bend left=35] node[yshift=-1.5ex] {$p$} (m-1-1)
(m-1-2) edge [loop right, in=35,out=-35,looseness=5, min distance=10mm] node {$\phi$} (m-1-2)
;
\end{tikzpicture}
\end{center}
where $i\colon F \to T$ is the inclusion and $p\colon T \to F$ is the map from (\ref{definitionp}). By Theorem \ref{merkulovtheorem}, we obtain an $A_{\infty}$-structure $\mu_n$ on $F$.
From Theorem \ref{munminimalimpliesgolod} it follows that it is sufficient to show that each $\mu_n$ is minimal. Recall that $\mu_n = p\lambda_n$ where
$$\lambda_n = \sum_{\substack{s+t=n \\ s,t \geq 1}} (-1)^{s+1} \lambda_2(\phi \lambda_s \otimes \phi \lambda_t).$$
Therefore, it is sufficient to prove that $p\lambda_2$ maps into the maximal ideal. Let $u_I$ and $u_J$ be basis elements of $T$. 
We may assume that $\gcd(m_I,m_J) =1$ since otherwise $p\lambda_2(u_I \otimes u_J) \in (x_1,\ldots,x_m)F$ by Lemma \ref{plambda2minimalonnondisjoint}.
Write $I = \lbrace i_1, \ldots, i_k \rbrace$ and $J = \lbrace i_{k+1}, \ldots, i_{n} \rbrace$ where $n=k+l$. 
By definition of $p$ we have
$$p(u_{i_1} \cdots u_{i_n}) = m[\sum_{\sigma \in S_n} \sgn(\sigma) \pi(u_{\sigma I_1}) \cdots \pi(u_{\sigma I_n}) ]$$
where $m = \lcm(m_I,m_J) = m_Im_J$ and $u_{\sigma I_p} = u_{i_{\sigma(1)}} \cdots u_{i_{\sigma(p)}}$.
Write 
$$\alpha_{\sigma} = \frac{m}{\lcm(\pi(m_{\sigma I_1}), \ldots, \pi(m_{\sigma I_n}))}$$
then
$$p(u_{i_1} \cdots u_{i_n}) = \sum_{\sigma \in S_n} \sgn(\sigma) \alpha_{\sigma} \pi(u_{\sigma I_1}) \cdots \pi(u_{\sigma I_n}).$$
We need to show that $\alpha_{\sigma} \in (x_1,\ldots,x_m)$ for all $\sigma \in S_n$. 
Suppose $\alpha_{\sigma} = 1$ for some $\sigma \in S_n$. 
Without loss of generality, we may assume $i_{\sigma(1)} \in I$. 
 Set
$$q = \min \lbrace q' \vert i_{\sigma(q')} \in J \rbrace.$$
By assumption, $\lcm(\pi(m_{\sigma I_1}),\ldots, \pi(m_{\sigma I_n}))$ is divisible by $m_{i_{\sigma(q)}}$. 
Since $\gcd(m_{i_{\sigma(q)}}, m_I) = 1$, we have $\gcd(m_{i_{\sigma(q)}}, \pi(m_{\sigma I_k})) =1$ for all $k<q$. Therefore, $\lcm(\pi(m_{\sigma I_q}), \ldots, \pi(m_{\sigma I_n}))$ is still divisible by $m_{i_{\sigma(q)}}$.

We claim that 
$$m_{i_{\sigma(q)}} \notin \lbrace \pi(m_{\sigma I_q}), \ldots, \pi(m_{\sigma I_n}) \rbrace.$$
Indeed, assume that $m_{i_{\sigma(q)}} = \pi(m_{\sigma I_s})$ for some $s \geq q$. We have that $\pi(m_{\sigma I_s}) = \pi(m_{i_{\sigma(1)}}, \ldots, m_{i_{\sigma(s)}})$. 
Then
$$m_{i_{\sigma(q)}} \vert \lcm(m_{i_{\sigma(1)}},m_{i_{\sigma(q)}}) \vert \lcm(m_{i_{\sigma(1)}}, \ldots, m_{i_{\sigma(s)}})$$
and so $m_{i_{\sigma(q)}} = \pi(m_{i_{\sigma(1)}},m_{i_{\sigma(q)}})$ since $\pi$ is a rooting map.
But by definition of $q$ we have $\gcd(m_{i_{\sigma(1)}},m_{i_{\sigma(q)}})=1$ so this contradicts $I$ being gcd by Lemma \ref{gcdpigcd}. 
Therefore
$$m_{i_{\sigma(q)}} \notin \lbrace \pi(m_{\sigma I_q}), \ldots, \pi(m_{\sigma I_n}) \rbrace.$$
Define
$$u =  u_{i_{\sigma(q)}} \pi(u_{\sigma I_q})\cdots \pi(u_{\sigma I_n})$$
then we claim that $u$ is in $F$. To see that $u$ is rooted, let $v \subseteq \{ u_{i_{\sigma(q)}}, \pi(u_{\sigma I_q}), \ldots, \pi(u_{\sigma I_n}) \}$.
If $u_{i_{\sigma(q)}} \notin v$ then there is nothing to prove as $\{\pi(u_{\sigma I_q}), \ldots, \pi(u_{\sigma I_n}) \}$ is rooted. So, assume $u_{i_{\sigma(q)}} \in v$. We can write
$$ v = u_{i_{\sigma(q)}} \pi(u_{\sigma I_{q_1}})\cdots \pi(u_{\sigma I_{q_k}})$$
for some $q_i \geq q$. We have 
$$ \pi(u_{\sigma I_{q_k}}) \vert m_v \vert m_{\sigma I_{q_k}}$$
and so $\pi(v) = \pi(u_{\sigma I_{q_k}}) \in v$. Hence, $u$ is rooted as claimed. But $du \notin (x_1,\ldots,x_m)F$ since $m_{i_{\sigma(q)}}$ divides $\lcm(\pi(m_{\sigma I_q}),\ldots, \pi(m_{\sigma I_n}))$ which contradicts minimality of $F$. 
\end{proof}

We now come to the main theorem of this section.

\begin{theorem}
\label{golodequicupzeroequipigcd}
 Let $R$ be a rooted ring. Then the following are equivalent.
 \begin{enumerate}
  \item The ring $R$ is Golod.
  \item The product on $\tor^S(R,k)$ vanishes.
  \item The ring $R$ is gcd. 
 \end{enumerate}
\end{theorem}
\begin{proof}
 The implication $1 \Rightarrow 2$ is immediate from the definition and $3 \Rightarrow 1$ follows by Lemma \ref{pigcdimpliesgolod}. 
 We prove $2 \Rightarrow 3$. Since the product on $\tor^S(R,k)$ is just $\mu_2 \otimes 1$, the product vanishes if and only if $\mu_2$ is minimal.
 Let $m_i$ and $m_j$ be generators such that $\gcd(m_i,m_j)=1$. Then
$$\mu_2(u_i, u_j) = \frac{\lcm(m_i,m_j)}{\lcm(\pi(m_i,m_j)m_i)} \pi(u_i,u_j)u_i - \frac{\lcm(m_i,m_j)}{\lcm(\pi(m_i,m_j)m_j)} \pi(u_i,u_j)u_j.$$
 If $\pi(m_i,m_j)=m_j$ then
$$\frac{\lcm(m_i,m_j)}{\lcm(\pi(m_i,m_j)m_j)} = 1$$
 which contradicts minimality of $\mu_2$ and so $\pi(m_i,m_j) \neq m_j$. 
 By the same argument, $\pi(m_i,m_j) \neq m_i$ and thus $R$ is gcd by Lemma \ref{gcdpigcd}.
\end{proof}

\begin{remark}
The equivalence between the second and third statement of Theorem \ref{golodequicupzeroequipigcd} is known. See for example Lemma 2.4 of \cite{katthan2015}
\end{remark}

\begin{example}
Let $S=k[x_1,\ldots,x_9]$ and let $I$ be the ideal 
$$(x_2x_5x_8, x_2x_3x_8x_9,x_5x_6x_7x_8,x_1x_2x_4x_5,x_1x_2x_3,x_4x_5x_6,x_7x_8x_9).$$
Label the generators by $u_1, \ldots, u_9$ and order them by $u_1 \prec u_2 \prec \cdots \prec u_9$. Let $L$ be the Lyubeznik resolution with respect to the ordering $\prec$. 
Then $L$ is easily seen to be minimal. Plainly, $I$ satisfies the gcd condition and so $S/I$ is Golod.
\end{example}

\section*{Acknowledgements}
The author would like to thank his PhD supervisor Jelena Grbi\'c for advice and guidance, 
Fabio Strazzeri and Francisco Belch\'{\i} for useful discussions 
and Lukas Katth\"an, Bernhard K\"ock, Taras Panov and the referee for helpful comments on an earlier version of this manuscript.

\bibliographystyle{plain}
\bibliography{master}
\end{document}